\documentclass{article}
\usepackage{graphicx} 
\usepackage{amsthm}
\usepackage{amsfonts, amsmath, amssymb, amsthm}
\usepackage{cleveref}
\usepackage[utf8x]{inputenc}

\author{Quentin Le Hou\'erou \and Ludovic Levy Patey}

\title{$\Pi^0_4$ conservation of the Ordered Variable Word theorem}
\date{August 2023}

\newtheorem{theorem}{Theorem}
\numberwithin{theorem}{section}

\newtheorem{lemma}[theorem]{Lemma}
\newtheorem{question}[theorem]{Question}
\newtheorem{proposition}[theorem]{Proposition}

\newtheorem{remark}[theorem]{Remark}
\newtheorem{definition}[theorem]{Definition}
\newtheorem{corollary}[theorem]{Corollary}
\newtheorem{example}[theorem]{Example}

\makeatletter
\newtheorem*{rep@theorem}{\rep@title}
\newcommand{\newreptheorem}[2]{%
\newenvironment{rep#1}[1]{%
 \def\rep@title{#2 \ref{##1}}%
 \begin{rep@theorem}}%
 {\end{rep@theorem}}}
\makeatother

\newreptheorem{theorem}{Theorem}

\usepackage{xcolor}

\newcommand{\RCA}{\mathsf{RCA}}
\newcommand{\WKL}{\mathsf{WKL}}

\newcommand{\ACA}{\mathsf{ACA}}

\newcommand{\BSig}{\mathsf{B}\Sigma^0}
\newcommand{\ISig}{\mathsf{I}\Sigma^0}

\newcommand{\RT}{\mathsf{RT}}

\newcommand{\TT}{\mathsf{TT}}
\newcommand{\MTT}{\mathsf{MTT}}
\newcommand{\OVW}{\mathsf{OVW}}
\newcommand{\GR}{\mathsf{GR}}

\renewcommand{\L}{\mathcal{L}}

\newcommand{\M}{\mathcal{M}}

\newcommand{\NN}{\mathbb{N}}

\newcommand{\uh}[0]{{\upharpoonright}}

\newcommand{\finsub}{\subseteq_{\mathtt{fin}}}

\makeatletter
\DeclareFontEncoding{LS2}{}{\@noaccents}
\makeatother
\DeclareFontSubstitution{LS2}{stix}{m}{n}
\DeclareSymbolFont{arrows3}{LS2}{stixtt}{m}{n}
\DeclareMathSymbol{\smashtimes}{\mathbin}{arrows3}{"A4}

\usepackage{mathabx}

\def\qt#1{``#1''}%

\begin{document}

\maketitle

\begin{abstract}
A left-variable word over an alphabet~$A$ is a word over~$A \cup \{\star\}$ whose first letter is the distinguished symbol~$\star$ standing for a placeholder.
The Ordered Variable Word theorem ($\OVW$), also known as Carlson-Simpson's theorem, is a tree partition theorem, stating that for every finite alphabet~$A$ and every finite coloring of the words over~$A$, there exists a word $c_0$ and an infinite sequence of left-variable words $w_1, w_2, \dots$
such that $\{ c_0 \cdot w_1[a_1] \cdot \dots \cdot w_k[a_k] : k \in \NN, a_1, \dots, a_k \in A \}$ is monochromatic.

In this article, we prove that $\OVW$ is $\Pi^0_4$-conservative over~$\RCA_0 + \BSig_2$. This implies in particular that $\OVW$ does not imply $\ACA_0$ over~$\RCA_0$. This is the first principle for which the only known separation from~$\ACA_0$ involves non-standard models.
\end{abstract}

\section{Introduction}

A \emph{tree partition theorem} is a statement of the form \qt{For every finite coloring of the finite subtrees of an infinite tree-like structure, there exists an isomorphic sub-structure whose finite subtrees are monochromatic.} Perhaps the simplest tree partition theorem is the Tree Theorem for singletons ($\TT^1)$ which says that for every finite coloring of~$2^{<\omega}$, there is a monochromatic subset~$T \subseteq 2^{<\omega}$ such that $(T, \preceq)$  is isomorphic to $(2^{<\omega}, \preceq)$.

Tree partition theorems play an important role in structural Ramsey theory. Many proofs of existence of \emph{big Ramsey numbers} are reduced to higher order versions of these theorems. For example, the existence of big Ramsey numbers for partitions of the rationals~\cite{devlin1980some} or of the Rado graph~\cite{sauer2006coloring} are both reduced to Milliken's tree theorem~\cite{milliken1979ramsey}. 

In this article, we are interested in tree partition theorems from the viewpoint of reverse mathematics. \emph{Reverse mathematics} is a foundational program whose goal is to find optimal axioms to prove ordinary theorems. It uses the framework of subsystems of second-order arithmetic, with a base theory, $\RCA_0$, capturing \qt{computable mathematics}. See any of~\cite{hirschfeldt2015slicing,dzhafarov2022reverse,simpson2009subsystems} for a good introduction to reverse mathematics, and their main systems, $\RCA_0, \WKL_0$ and $\ACA_0$.

Both the Tree Theorem ($\TT$) and Milliken's tree theorem ($\MTT$) have been extensively studied from a reverse mathematical viewpoint (see~\cite{chong2020strength,chong2023conservation,chubb2009reverse,corduan2010reverse,dzhafarov2017coloring,patey2016strength} for $\TT$ and \cite{angles2020milliken} for $\MTT$).  The restrictions of $\TT$ and $\MTT$ to colorings of singletons are \emph{computably true}, that is, every instance admits a solution computable in the instance. Thus, their strength can be measured only in terms of the amount of induction necessary to prove them.

\subsection{Ordered Variable Word theorem}

In this article, we study the reverse mathematics of a stronger tree partition theorem, due to Carlson and Simpson~\cite{carlson1984dual}, called the Ordered Variable Word theorem~\footnote{This statement is also called \emph{Carlson-Simpson theorem} in combinatorics~\cite{dodos2014density}, but should be distinguished from the Variable Word theorem~\cite{miller2004effectiveness}, also known as Carlson-Simpson lemma in reverse mathematics, which is a similar statement, but where each variable type is allowed to appear infinitely often.}.
Fix a finite alphabet~$A$ and a distinguished variable symbol~$\star$.
A \emph{word} over~$A$ is a finite sequence $w = a_0\dots a_{k-1}$ where $a_i \in A$. We write $w(i)$ for $a_i$ and let $|w| = k$. 
A \emph{left variable word} over~$A$ is a word $w$ over~$A \sqcup \{\star\}$ such that~$w(0) = \star$. Given a left variable word~$w$ and a letter~$a \in A$, we let $w[a]$ be the word of length~$|w|$ obtained from~$w$ by substituting every occurrence of~$\star$ by~$a$. Beware, the notations $w(i)$ and $w[a]$ should not be confused: the former yields a letter, while the latter is a substitution.

\begin{theorem}[Carlson-Simpson~\cite{carlson1984dual}]
For every finite alphabet~$A$ and every finite coloring $f : A^{<\omega} \to \ell$,
there is a word~$c_0$ over~$A$ and an infinite sequence of left variable words $w_1, w_2, \dots$ such that the following set is monochromatic:
$$
\{ c_0 w_1[a_1] w_2[a_2] \cdots w_k[a_k] : k \in \NN, a_1, \dots, a_k \in A \}
$$
\end{theorem}

We let~$\OVW$ denote the statement above, standing for Ordered Variable Word. It can be seen as a problem whose \emph{instances} are pairs $(A, f)$, where $A$ is a finite alphabet and $f$ is a finite coloring over~$A^{<\omega}$. A \emph{solution} to $(A, f)$ is the given of the sequence $c_0, w_1, w_2, \dots$ as above.

The Ordered Variable Word theorem was used as a pigeonhole principle by Carlson and Simpson to prove a dual version of Ramsey's theorem~\cite{carlson1984dual}. It was later used by Hubi\v{c}ka~\cite{hubivcka2020big} to give a simple proof of the existence of big Ramsey degrees for the universal triangle-free graph. See the monograph of Dodos and Kanellopoulos~\cite{dodos2016ramsey} for an extensive combinatorial analysis of variable word theorems.
Friedman and Simpson~\cite{friedman2000issues} asked about its reverse mathematical strength.

Contrary to the Tree Theorem for singletons and Milliken's tree theorem for singletons which are computably true, Miller and Solomon~\cite{miller2004effectiveness} proved that $\OVW$ admits a computable instance with no $\Delta^0_2$ solutions. Therefore, $\OVW$ cannot be proven by Weak K\"onig's lemma ($\WKL_0$). Later, Liu, Monin and Patey~\cite{liu2019computable} constructed a computable instance of~$\OVW$ whose solutions compute a DNC function relative to~$\emptyset'$, that is, a function $f : \NN \to \NN$ such that $\forall e f(e) \neq \Phi^{\emptyset'}_e(e)$. On the other hand, Angl\`es d'Auriac et al.~\cite{angles2023carlson} proved that every computable instance of $\OVW$ admits a solution in any PA degree over~$\emptyset'$. The exact computable strength of $\OVW$ therefore lies between DNC degrees over~$\emptyset'$ and PA degrees over~$\emptyset'$. Note that the existence of the former is strictly weaker than~$\ACA_0$, while the latter implies it. Whether~$\OVW$ implies~$\ACA_0$ over~$\RCA_0$ was left open. We answer the question negatively:

\begin{theorem}
$\WKL_0 + \OVW$ does not imply~$\ACA_0$.
\end{theorem}

Usually, separations from~$\ACA_0$ are done using the so-called \emph{cone avoidance} property, that is, by proving that for every non-computable set~$A$, every computable instance of the problem admits a solution which does not compute~$A$. In this article, we take a different approach, and prove the separation through a partial conservation result.

\subsection{First-order parts and partial conservation}

A good way to get a better grasp on the reverse mathematical strength of a second-order theorem is to understand its \emph{first-order part}, that is, the theory of its first-order consequences. For example, the first-order part of $\ACA_0$ is Peano arithmetic, while the first-order parts of~$\RCA_0$ and $\WKL_0$ both correspond to~$\Sigma_1$-PA, the restriction of PA to $\Sigma_1$-induction.
There exist two main families of first-order systems which are good benchmarks of the first-order strength of a theorem: induction and collection principles.

\begin{definition}Let $\Gamma$ be a class of formulas (e.g. $\Sigma^0_n, \Pi^0_n$)
\begin{itemize}
    \item The \emph{induction} scheme $\mathsf{I}\Gamma$ is defined for every formula $\varphi \in \Gamma$ by:
    $$
    \varphi(0) \wedge \forall x(\varphi(x) \to \varphi(x+1)) \rightarrow \forall x \varphi(x)
    $$
    \item The \emph{collection} scheme $\mathsf{B}\Gamma$ is defined for every formula $\varphi \in \Gamma$:
    $$
    \forall a [(\forall x < a)(\exists y)\varphi(x, y) \rightarrow (\exists b)(\forall x < a)(\exists y < b)\varphi(x, y) ]
    $$
\end{itemize}
\end{definition}

These schemes form a strictly increasing hierarchy as follows: 
$$
\ISig_1 < \BSig_2 < \ISig_2 < \BSig_3 < \dots
$$
The collection scheme can be thought of as an induction scheme: $\BSig_n$ is equivalent over~$\RCA_0$ to the induction scheme for $\Delta^0_n$ predicates~\cite{slaman2004bounding}. The collection scheme for $\Sigma^0_2$ is of particular interest, and admits many characterizations. In combinatorics, $\BSig_2$ is equivalent to $\RT^1$, the infinite pigeonhole principle~\cite{cholak2001strength}. The tree theorem for singletons is strictly in between~$\ISig_2$ and $\BSig_2$ over~$\RCA_0$ (see~\cite{corduan2010reverse,chong2020strength}). It is unknown whether $\TT^1$ is $\Pi^1_1$-conservative over~$\RCA_0 + \BSig_2$. 

Chong, Wang and Yang~\cite{chong2023conservation} obtained a partial conservation result by proving that $\TT^1$ is $\forall \Pi^0_3$-conservative over~$\RCA_0$. Here, a $\forall \Pi^0_n$ formula consists in a universal set quantification followed by a $\Pi^0_n$ formula.
We prove the following partial conservation theorem for the Ordered Variable Word theorem, which strengthens Chong, Wang and Yang's result since~$\TT^1$ follows from~$\OVW$:

\begin{theorem}\label[theorem]{thm:OVW-pi04-conservation}
$\WKL_0 + \OVW$ is $\forall \Pi^0_4$-conservative over $\RCA_0 + \BSig_2$.      
\end{theorem}

Note that Parsons, Friedman and Paris proved that $\BSig_{n+1}$ is $\forall \Pi^0_{n+2}$-conservative over~$\ISig_n$ for every~$n \geq 0$ (see \cite{hajek1998metamathematics,kaye1991models}), but not $\Pi^0_{n+3}$-conservative since $\mathsf{B}\Sigma_{n+1}$ is a $\Pi_{n+3}$ sentence.
Note that this result cannot be strengthened to conservation over~$\RCA_0$ since~$\BSig_2$ is not $\Pi^0_4$-conservative over~$\RCA_0$. Also note that this conservation results yields a separation of~$\OVW$ from~$\ACA_0$, since~$\ACA_0$ proves the consistency of~$\RCA_0$ (see Simpson~\cite[Corollary VIII.1.7]{simpson2009subsystems}).

\subsection{Organization of the paper}

In \Cref{sect:largeness}, we introduce the framework used to prove the partial conservation result. Then, in \Cref{graham-rothschild}, we prove the existence of some largeness bounds for the Graham-Rothschild theorem, which are then used in \Cref{sect:ovw} to prove the existence of largeness bounds for the Ordered Variable Word theorem.  Last, in \Cref{sect:consequences}, we derive some consequences of the proof and state some remaining open questions.

\section{Largeness}\label[section]{sect:largeness}

A collection $\L$ of finite sets is a \emph{notion of largeness} if it is closed under superset, and every infinite set has a finite subset in~$\L$.
Ketonen and Solovay~\cite{ketonen1981rapidly} defined a quantitative notion of largeness, called $\alpha$-largeness, to better understand the unprovability of combinatorial principles such as Ramsey's theorem in some theories. This was later combined by Patey and Yokoyama~\cite{patey2016strength} with the indicator argument techniques of Paris and Kirby~\cite{Kirby1977InitialSO} to prove $\forall \Pi^0_3$-conservation results about Ramsey's theorem for pairs. More recently, Le Houérou, Levy Patey and Yokoyama~\cite{houerou2023conservation} introduced a variant of $\alpha$-largeness to prove $\forall \Pi^0_4$-conservation results about~$\RT^2_2$. 

In what follows, we work in a language enriched with two constant symbols $c$ and $C$ representing a first-order and second-order object respectively. Fix a $\Delta_0^0$ formula $\zeta(X\uh_z,t,x,y,z)$ with exactly the displayed free variables, and let $\theta(x, y, z) = \zeta(C\uh_z,c,x,y,z)$ be a $\Delta^{0,C,c}_0$ formula. Given two sets~$A$ and $B$, we write $A < B$ for $(\forall a \in A)(\forall b \in B)a < b$.

\begin{definition}
Two finite sets~$X < Y$ are \emph{$\theta$-apart} if 
$$\forall x < \max X \exists y < \min Y \forall z < \max Y \theta(x, y, z)$$
\end{definition}

Note that $\theta$-apartness is a transitive relation. Moreover, if $X < Y$ are $\theta$-apart and $X_0 \subseteq X$ and $Y_0 \subseteq Y$, then $X_0, Y_0$ are $\theta$-apart.

\begin{definition}[\cite{houerou2023conservation}]\label[definition]{defi:largeness-rca0-bsig2-t}
A set~$X \finsub \NN$ is 
\begin{itemize}
    \item \emph{$\omega^0$-large$(\theta)$} if $X \neq \emptyset$.
    \item \emph{$\omega^{(n+1)}$-large$(\theta)$} if $X \setminus \min X$ is  $(\omega^n \cdot \min X)$-large$(\theta)$
    \item \emph{$\omega^n \cdot k$-large$(\theta)$} if
there are $k$ pairwise $\theta$-apart $\omega^n$-large$(\theta)$ subsets of~$X$
$$
X_0 < X_1 < \dots < X_{k-1}
$$
\end{itemize}
\end{definition}

Note that if we take $\theta(x, y, z)$ to be the $\top$ formula, then $\omega^n \cdot k$-largeness$(\theta)$ is exactly $\omega^n \cdot k$-largeness. 
The following proposition was proven by Le Houérou, Levy Patey and Yokoyama~\cite{houerou2023conservation}.

\begin{proposition}[\cite{houerou2023conservation}]\label[proposition]{prop:largeness-bsig2-largeness}
For every~$n \in \omega$,
$\RCA_0 + \BSig_2 + \forall x \exists y \forall z \theta(x, y, z)$ proves that for every~$b \geq 1$, $\omega^n \cdot b$-largeness$(\theta)$ is a largeness notion.
\end{proposition}


\begin{remark}
Many combinatorial proofs about largeness make some assumptions on $\min X$ in order to avoid some degenerate cases.
For example, Ketonen and Solovay~\cite{ketonen1981rapidly} and Ko{\l}odziejczyk and Yokoyama~\cite{kolo2020some} assumed that $\min X \geq 3$.
In this article, because of \Cref{lem:split-largeness} and \Cref{lem:large-implies-sparse}, we will always assume that $\min X \geq 4$. As in Le Houérou, Levy Patey and Yokoyama~\cite{houerou2023conservation} we will also require that $\min X \geq c$, for technical reasons which will become clear in the proof of \Cref{thm:OVW-pi04-conservation}.
\end{remark}

In the remainder of this section, we prove some basic combinatorial lemmas about $\omega^n$-largeness$(\theta)$ which will be used throughout this article. These lemmas could be considered as folklore, in the sense that their $\alpha$-largeness counterpart are well-known, and their adaptation to $\omega^n$-largeness$(\theta)$ is almost straightforward, except maybe \Cref{lem:split-largeness}. The proof of the next lemma is very similar to \cite[Lemma 2.9]{houerou2023conservation}

\begin{lemma}\label[lemma]{lem:split-largeness}
$\ISig_1$ proves that for every $a,b$ and for every~$\omega^{a+2b+1}$-large$(\theta)$ set~$X$, then there are some $k \in \NN$ and some $\omega^a$-large$(\theta)$ pairwise $\theta$-apart subsets $X_0 < \dots < X_{k-1}$ of~$X$ such that every $H \in \prod_{i < k} X_i$ is $\omega^b$-large$(\theta)$ (here, we abuse the product notation and see $H$ as a subset of $\NN$ intersecting every $X_i$ exactly once rather than a tuple in the product $X_0 \times \dots \times X_{k-1}$).

\end{lemma}

\begin{proof}
By $\Pi^0_1$-induction over $b$, we prove the following statement that directly imply the lemma (since $\min X \geq 2$): for all $a$, for all $\theta$-apart pairs $Y_0 < Y_1$ of $\omega^{a+2b}$-large$(\theta)$ sets, there exists some $k < \max Y_1$ and some $\omega^a$-large$(\theta)$ subsets $X_1 < \dots < X_{k-1}$ of $Y_1$ such that, letting $X_0 = Y_0$, every $H \in \prod_{i < k} X_i$ is $\omega^b$-large$(\theta)$ and $X_0, \dots, X_{k-1}$ are pairwise $\theta$-apart. This is indeed a $\Pi_1^0$ formula, as being $\theta$-appart, or being $\omega^{a}$-large$(\theta)$ are $\Delta_0^0$ statements.

Case $b = 0$. The result is clear, as every non-empty set is $\omega^0$-large$(\theta)$

Case $b > 0$. Let $Y_0 < Y_1$ be $\omega^{a+2b}$-large$(\theta)$ and $\theta$-apart, let $Z_0 < \dots < Z_{\min Y_1 - 1}$ be $\omega^{a+2b-1}$-large$(\theta)$ pairwise $\theta$-apart subsets of $Y_1$. For every $i < \min Y_1$, let $Z_i^0, Z_i^1$ be $\omega^{a+2b-2}$-large$(\theta)$ and $\theta$-apart subsets of $Z_i$. 

We can then apply the inductive hypothesis on the pairs $$(Z_0^0,Z_0^1), (Z_1^0,Z_1^1) \dots, (Z_{\max Y_0 - 1}^0, Z_{\max Y_0 - 1}^1)$$
to get for every $j < \max Y_0$, families of pairwise $\theta$-apart $\omega^{a}$-large$(\theta)$ subsets $Z_{j}^0 = X_{j,0} < \dots < X_{j,k_j}$ of $Z_{j}^0 \cup Z_{j}^1$ such that every $H \in \prod_{i \leq k_j} X_{j,i}$ is $\omega^{b-1}$-large$(\theta)$.

Consider the family, $Y_0 < X_{0,0} < X_{0,1} < \dots < X_{0,k_0} < X_{1,0} < X_{1,1} < \dots < X_{\max Y_0 - 1, k_{\max Y_0 - 1}}$. Every block of this family is $\omega^a$-large$(\theta)$. Since $X_{0,0} \subseteq Y_1$, then $Y_0$ and $X_{0,0}$ are $\theta$-apart. Moreover, for all $j < \max Y_0 - 1$, since $X_{j,k_j} \subseteq Z_{j}$ and $X_{j+1,0} = Z_{j+1}$, $X_{j,k_j}$ and $X_{j+1,0}$ are $\theta$-apart.

Let $H \in Y_0 \times X_{0,0} \times \dots \times X_{0,k_0} \times X_{1,0} \times \dots \times X_{\max Y_0 - 1, k_{\max Y_0 - 1}}$. Every element of $\prod_{i < k_j} X_{j,i}$ is $\omega^{b-1}$-large$(\theta)$, therefore $H$ is $\omega^{b}$-large$(\theta)$.

This completes the proof.
\end{proof}

\begin{remark}\label[remark]{rem:lower-bound-blocks}
The bound obtained in \Cref{lem:split-largeness} is tight in the sense that it is not possible to get rid of the $2m$ in the exponent: there exists a formula $\theta$ and a set $X$ that is $\omega^{2m-1}$-large$(\theta)$ such that no family $X_0 < \dots < X_{k-1}$ of $\omega$-large$(\theta)$ and pairwise $\theta$-apart subsets of $X$ satisfy that every $H \in \prod_{i < k} X_i$ is $\omega^m$-large$(\theta)$.

To see this, we can use the construction in \cite{houerou2023conservation} of a formula $\theta$ and of an $\omega^{2m-1}$-large$(\theta)$ set $X$ with a coloring $f : X \to 2$ such that no $f$-homogeneous $\omega^m$-large$(\theta)$ subset exists. The coloring $f$ has the property that for any $\omega$-large$(\theta)$ subset $Y$ of $X$, there exists some $y_0,y_1 \in [\min Y, \max Y] \cap X$ such that $f(y_0) = 0$ and $f(y_1) = 1$.

Assuming there exists a family $X_0 < \dots < X_{k-1}$ of $\omega$-large$(\theta)$ subsets of $X$ such that every $H \in \prod_{i < k} X_i$ is $\omega^m$-large$(\theta)$, we get that every $H \in \prod_{i < k} [\min X_i, \max X_i] \cap X$ is $\omega^m$-large$(\theta)$ ($\theta$ has the property that the $\theta$-apartness of two sets $Y_0$ and $Y_1$ depends only on $\max Y_0$, $\min Y_1$ and $\max Y_1$) and therefore, by taking in each set $[\min X_i, \max X_i] \cap X$ an element $h_i$ such that $f(h_i) = 0$, we obtain a set $H$ that is $f$-homogeneous and $\omega^m$-large$(\theta)$, contradicting the properties of $X$. \\
\end{remark}

\begin{remark}
Contrary to \Cref{rem:lower-bound-blocks}, when considering the classical notion of largeness, it is possible to show that for every $\omega^{n+m+1}$-large set $X$, there exists $\omega^{n}$-large subsets $X_0 < \dots < X_{k-1}$ such that every $H \in \prod_{i < k} X_i$ is $\omega^m$-large. Propagating this difference between the bound obtained for largeness and for largeness$(T)$, the bound to obtain an $\omega^n$-large$(\OVW)$ set is linear, while the bound to obtain an $\omega^n$-large$(\theta, \OVW)$ set is exponential. 
\end{remark}

\begin{lemma}[Folklore, see Ketonen and Solovay~\cite{ketonen1981rapidly}]\label[lemma]{lem:fast-growing-size}
For every primitive recursive function $f$, there exists some $n \in \omega$ such that every $\omega^n\mbox{-large}$ (and a fortiori $\omega^n\mbox{-large}(\theta)$) set $X$ satisfies $|X| > f(\min X - 1)$.
\end{lemma}



Ko{\l}odziejczyk and Yokoyama~\cite{kolo2020some} defined two notions of sparsity, called \emph{exp-sparsity} and \emph{$\alpha$-sparsity}, respectively, and proved that with a constant overhead on the bounds of largeness, one could always assume that the set is sufficiently sparse. We define a similar notion of sparsity, stronger than exp-sparsity, and prove the corresponding lemma.

\begin{definition}[Sparsity]
A set $X$ is said to be \emph{sparse} if for every $x,y \in X$ with $x < y$ we have $x^{x^x} < y$.
\end{definition}

It is clear that if $\RCA_0$ proves that some $\Gamma$ is a largeness notion, then $\RCA_0$ proves that being $\Gamma$-large and sparse is again a largeness notion. Indeed, given any infinite set~$X$, $\RCA_0$ proves the existence of an infinite sparse subset~$Y \subseteq X$. By largeness of~$\Gamma$, there is a finite $\Gamma$-large subset~$Z \subseteq Y$, which is both $\Gamma$-large and sparse. 

\begin{lemma}\label[lemma]{lem:large-implies-sparse}
If $X$ is $\omega^{2n+7}\mbox{-large}(\theta)$, then there is a subset $Y \subseteq X$ that is $\omega^n\mbox{-large}(\theta)$ and sparse.    
\end{lemma}

\begin{proof}
By \Cref{lem:split-largeness}, there exists $X_0 < \dots < X_{k-1}$ $\omega^3\mbox{-large}(\theta)$ subsets of $X$ such that $\{\max X_i : i < k\}$ is $\omega^n\mbox{-large}(\theta)$.

This set is also sparse: by a simple computation, if $(x,y]$ is $\omega^3\mbox{-large}(\theta)$ then $y > x^{x^x}$ (using the assumption that $\min X \geq 4$).   
\end{proof}

\subsection{Largeness and variable words}

Largeness is defined in terms of sets of integers, while the Ordered Variable Word theorem is a statement about words and variable words. We bridge the two notions by defining an ordered $X$-variable word over a finite set~$X \subseteq \NN$.

\begin{definition}[Ordered $Y$-variable word]
For $Y =\{y_0, \dots, y_{n-1}\}$ a finite set, an \emph{ordered $Y$-variable word over $A$} is a finite word $w$ over the alphabet $A \sqcup \{x_j : j < n\}$ where for every $j < n$, the first occurrence of $x_j$ is at position $y_j$ and for every $j < n-1$ its last occurrence is before position $y_{j+1}$.    
\end{definition}

\begin{definition}[Substitution]
For $w$ a $Y$-variable word over an alphabet $A$ and $u \in A^{\leq |Y|}$, the substitution $w[u]$ is defined as the word $w$ where every occurrence of $x_i$ for $i < |u|$ is replaced by $u(i)$ and cut just before the first occurrence of $x_{|u|}$.
\end{definition}

\begin{example}
On the alphabet $A = \{a,b\}$, $w = abx_0ax_0bx_1bb$ is a $\{2,6\}$-variable word, $ax_0bx_1x_0ab$ is not a variable word, since there is an occurrence of $x_0$ after an occurrence of $x_1$ and $\ aax_1b$ is not a variable word either since there is no occurrence of $x_0$. 
$w[\epsilon] = ab$ (where $\epsilon$ is the empty word), $w[a] = abaaab$ and $w[ba] = abbabbabb$.
\end{example}

\section{Graham-Rothschild theorem}\label[section]{graham-rothschild}

The Ordered Variable Word theorem belongs to a whole family of variable words statements, whose pigeonhole principle is the \emph{Hales-Jewett theorem}~\cite{hales1963regularity}. In its simplest form, the Hales-Jewett theorem asserts the existence, for every finite alphabet~$A$ and every finite coloring of~$A^{<\omega}$, of a variable word~$w$ such that $\{ w[a] : a \in A \}$ is monochromatic. Its generalization to multiple dimensions is known as the Graham-Rothschild theorem~\cite{graham2013ramsey}. In this section, we define notions of largeness for these theorems, and obtain explicit quantitative bounds for them. This analysis will be reused in the next section to obtain quantitative bounds for the Ordered Variable Word theorem.

\begin{definition}
For~$Y$ a set, a combinatorial $Y$-space over an alphabet $A$ is a set of the form $S = \{w[u] : u \in A^{|Y|}\}$ for some ordered $Y$-variable word $w$ over~$A$. We call $w$ its \emph{generating variable word}. 
The \emph{dimension} of a combinatorial $Y$-space $T$ over~$A$ is the number of variable kinds of its generating variable words. 
A \emph{combinatorial $Y$-line} is a combinatorial $Y$-space of dimension~1. \\
\end{definition}

\begin{definition}
For $S$ a combinatorial $X$-space, a \emph{combinatorial subspace} $S'$ of $S$ is a combinatorial $Y$-space included in $S$ for some set $Y$. 
\end{definition}

A consequence of the definition is that if $S'$ is a combinatorial $Y$-subspace of a combinatorial $X$-space~$S$, then $Y \subseteq X$ and the generating word $w'$ of $S'$ is equal to $w[u]$ for $w$ the generating word of $S$ and $u$ some $I$-variable word of length $|X|$, where $X = \{ x_0, \dots, x_n \}$ and $Y = \{ x_i : i \in I \}$.

The original proof of the following theorem by Hales and Jewett~\cite{hales1963regularity} had an extremely fast-growing bound. A more recent proof by Shelah~\cite{shelah1988primitive} uses only $\Sigma^0_1-$-induction, and yields a primitive recursive bound, hence is provable in $\RCA_0$.


\begin{theorem}[Hales-Jewett~\cite{hales1963regularity}, Shelah~\cite{shelah1988primitive}, $\RCA_0$]\label[theorem]{thm:hales-jewett}
There exists a primitive recursive function~$HJ(k, \ell)$ such that
for every set~$X$ with $|X| \geq HJ(k, \ell)$, every combinatorial $X$-space~$S$ over an alphabet~$A$ of size~$k$
and every coloring $f : S \to \ell$, there exists some~$x \in X$ and an $f$-homogeneous combinatorial $\{x\}$-subspace of~$S$.
\end{theorem}

As mentioned above, the Hales-Jewett theorem admits a multidimensional generalization, known as the Graham-Rothschild theorem~\cite{graham2013ramsey}. Its proof follows from the Hales-Jewett theorem by elementary combinatorics.
The original Graham-Rothschild theorem allows the variable kinds to be unordered, that is, that the first occurrence of~$x_i$ must appear before the first occurrence of~$x_{i+1}$, but the last occurrence of~$x_i$ might appear later. We are going to use a slightly modified version of the Graham-Rothschild theorem, due to Dodos et al~\cite[Theorem 2.1]{dodos2014density}, which requires the variable words to be ordered. The proof of its primitive recursive bound is an adaptation of Shelah's bound~\cite{shelah1988primitive} for the original Graham-Rothschild theorem, which can be found in Dodos and Kanellopoulos~\cite[Theorem 2.9, Theorem 2.15]{dodos2016ramsey}.


\begin{theorem}[Graham-Rothschild~\cite{graham2013ramsey}, Dodos and Kanellopoulos~\cite{dodos2016ramsey}, $\RCA_0$]\label[theorem]{thm:graham-rothschild}
There exists a primitive recursive function~$GR(k, d, m, \ell)$ such that for every set~$X$ with $|X| \geq GR(k, d, m, \ell)$, every combinatorial $X$-space~$S$ over an alphabet~$A$ of size~$k$ and every coloring $f$ of the $m$-dimensional subspaces of $S$ with $\ell$ colors, there exists some $d$-dimensional subspace of $S$, all of whose $m$-dimensional subspaces have the same color.
\end{theorem}

\begin{definition}
A set $X \subseteq_{fin} \NN$ is said to be $\omega^r \cdot s\mbox{-large}(\theta, \GR)$ if for every $\ell,k < \min X$, every combinatorial $X$-space over an alphabet $A$ of size $k$ and every coloring $f : S \to \ell$, there exists some $\omega^r \cdot s\mbox{-large}(\theta)$ subset $Y \subseteq X$ and an $f$-homogeneous combinatorial $Y$-subspace of $S$.
\end{definition}

For the following lemma, recall that a set $X$ is $\omega^0$-large$(\theta)$ iff $X \neq \emptyset$.
It therefore states that if $X$ is sufficiently large, then the size of~$X$ will be large enough to be able to apply the  Hales-Jewett theorem for $\min X - 1$ colors and alphabet of size~$\min X - 1$, and get a monochromatic combinatorial line.

\begin{lemma}\label[lemma]{lem:hales-jewett}
There exists some $n_0 \in \omega$ such that $\RCA_0$ proves that if $X$ is $\omega^{n_0}\mbox{-large}(\theta)$ then $X$ is $\omega^0\mbox{-large}(\theta,\GR)$.  
\end{lemma}

\begin{proof}
For $X$ to be $\omega^0\mbox{-large}(\theta,\GR)$, it is sufficient that $|X| \geq HJ(\min X - 1, \min X - 1)$. By \Cref{thm:hales-jewett}, $x \mapsto HJ(x,x)$ is primitive recursive, so by \Cref{lem:fast-growing-size}, there exists some $n_0$ such that every $\omega^{n_0}\mbox{-large}(\theta)$ set~$X$ satisfies $|X| \geq HJ(\min X - 1, \min X - 1)$.
\end{proof}

In the following lemma, it might be helpful for the reader to see a combinatorial $X$-space over an alphabet~$A$ as the set of leaves of a tree isomorphic to $A^{\leq |X|}$.

\begin{lemma}\label[lemma]{lem:gr-largeness}
$\RCA_0$ proves that for all $b \in \NN$ and every finite set~$X$, if $X$ is $\omega^{2b + n_0 + 1}\mbox{-large}(\theta)$ and sparse then $X$ is $\omega^b\mbox{-large}(\theta,\GR)$. 
\end{lemma}

\begin{proof}
Fix some~$b \in \NN$ such that the hypothesis holds.
Let $X$ be $\omega^{2b + n_0 + 1}\mbox{-large}(\theta)$, let $\ell,k < \min X$, let $S$ be a combinatorial $X$-space over an alphabet $A$ of size $k$ with generating variable word $w$, and let $f : S \to \ell$ be a coloring.

By \Cref{lem:split-largeness}, there exists pairwise $\theta$-apart $\omega^{n_0}\mbox{-large}(\theta)$ subsets of $X \setminus \{\min X\}$  $X_0 < \dots < X_{a-1}$ such that any $Y \in \prod_{i < a} X_i$ is $\omega^{b}\mbox{-large}(\theta)$. For simplicity, we will assume that $S$ is a combinatorial $(X_0 \cup \dots \cup X_{a-1})$-space (this can always be done by taking a subspace of the original space). \\

The coloring $f$ induce a coloring on every combinatorial $X_{a-1}$-subspace of~$S$. Let $n_{a-1}$ be the numbers of such subspaces, there is one of them for every possible value taken by the variables of $X_0 \cup \dots \cup X_{a-2}$, so 
$$n_{a-1} = k^{|X_0| + \dots + |X_{a-2}|} < k^{\max X_{a-2}}$$ 
We can consider the product of all those colorings $f' : S' \to \ell^{n_{a-1}}$ for $S'$ an arbitrary combinatorial $X_{a-1}$-space. By sparsity of $X$, 
$$\ell^{n_{a-1}} < \ell^{k^{\max X_{a-2}}} < \min X_{a-1}$$

And since $X_{a-1}$ is $\omega^{n_0}\mbox{-large}(\theta)$, by \cref{lem:hales-jewett}, there is some $y_{a-1} \in X_{a-1}$ and an $f$-homogeneous combinatorial $\{y_{a-1}\}$-subspace of $S'$. By definition of $f'$, for every combinatorial $X_{a-1}$-subspace of $S$, there is an $f$-homogeneous combinatorial $\{y_{a-1}\}$-subspace of it. \\


Consider $S_{a-1}$ an arbitrary combinatorial $(X_0 \cup \dots \cup X_{a-2})$-subspace of $S$ (for example, by fixing an arbitrary value for the variables of $X_{a-1}$). There is a one-to-one correspondence between the elements of $S_{a-1}$ and the combinatorial $X_{a-1}$-subspaces of $S$: every element of $S_{a-1}$ correspond to a possible value for the variables of $X_0 \cup \dots \cup X_{a-2}$ and therefore belongs to only one $X_{a-1}$-subspace of $S$ and vice versa, every $X_{a-1}$-subspace of $S$ contains only one element of $S_{a-1}$. So consider $f_{a-1} : S_{a-1} \to \ell$ a coloring that take any element of $S_{a-1}$ to the color of the $f$-homogeneous combinatorial $\{y_{a-1}\}$-subspace of the corresponding combinatorial $X_{a-1}$-subspace. 

We can then repeat the same argument and find some $y_{a-2} \in X_{a-2}$ such that, for every combinatorial $X_{a-2}$-subspace of $S_{a-1}$ there is an $f_{a-1}$-homogeneous combinatorial $\{y_{a-2}\}$-subspace of it. Therefore, for every combinatorial $(X_{a-2} \cup X_{a-1})$-subspace of $S$, there is an $f$-homogeneous combinatorial $\{y_{a-2}, y_{a-1}\}$-subspace of it.\\

Iterate the construction 
to find a sequence $y_0 \in X_0, \dots, y_{a-1} \in X_{a-1}$ such that, by letting $Y = \{y_0, \dots, y_{a-1}\}$, there is an $f$-homogeneous combinatorial $Y$-subspace of $S$. And $Y$ is $\omega^b\mbox{-large}(\theta)$ since $Y \in \prod_{i < a} X_i$.



   
\end{proof}

\section{Ordered Variable word theorem}\label[section]{sect:ovw}

We now turn to the quantitative analysis of the target theorem of this article, namely, the Ordered Variable Word theorem.
The main result of this section is \Cref{cor:largeness-OVW}, which is then used in \Cref{sect:consequences} to prove our conservation theorem.

\begin{definition}[Ordered variable word $Y$-tree] 
For $Y$ a set, an \emph{OVW $Y$-tree} over an alphabet $A$ is a set of the form $T = \{w[u] : u \in A^{\leq|Y|}\}$ for some ordered $Y$-variable word $w$ over~$A$. We call $w$ its \emph{generating variable word}. 
The \emph{dimension} of an OVW $Y$-tree $T$ over~$A$ is the number of variable kinds of its generating variable words, or equivalently the least~$n \in \omega$ such that $T$ is isomorphic to~$A^{\leq n}$.
A \emph{OVW $Y$-line} is an OVW $Y$-tree of dimension 1.
\end{definition}

\begin{definition}
For $T$ an OVW $X$-tree, an OVW subtree $T'$ of $X$ is an OVW $Y$-tree included in $T$ for some set $Y$.     
\end{definition}

A consequence of the definition is that if $T'$ is an OVW $Y$-subtree of an OVW $X$-tree~$T$, then $Y \subseteq X$ and the generating word $w'$ of $T'$ is equal to $w[u]$ for $w$ the generating word of $T$ and $u$ some $I$-variable word, where $X = \{ x_0, \dots, x_n \}$ and $Y = \{ x_i : i \in I \}$. We shall mainly consider two kinds of $Y$-subtrees: subtrees obtained by instantiating some variables of the generating word, which doesn't change its length, and subtrees obtained by truncating the generating word. \\

The following theorem is a finitary version of an infinitary theorem due to Carlson and Simpson~\cite{carlson1984dual}. The finitary version follows from its infinite version by compactness. Dodos et al.~\cite[Theorem 4.1]{dodos2014density} gave an elementary proof of a higher order variant of its finitary version, with a primitive recursive bound. The finitary version of the Ordered Variable word follows easily from the Graham-Rothschild theorem, and thus can almost be considered as folklore. The proof below appears in Dodos and Kanellopoulos~\cite[Proposition 4.10]{dodos2016ramsey}.

\begin{theorem}[Carlson-Simpson~\cite{carlson1984dual}, Dodos and Kanellopoulos.~\cite{dodos2016ramsey}, $\RCA_0$]\label[theorem]{thm:carlson-simpson-finitaire}
There exists a primitive recursive bound $CS(k,d,\ell)$ such that for every $X$ with $|X| \geq CS(k,d,\ell)$, every OVW $X$-tree $T$ over an alphabet $A$ of size $k$ and every coloring $f : T \to \ell$ there exists some subset $Y \subseteq X$ with $|Y| = d$ and an $f$-homogeneous OVW $Y$-subtree of $T$. 
\end{theorem}
\begin{proof}
Fix a set~$X$ such that $|X| \geq GR(k, d+1, 1, \ell)$.
Fix an OVW $X$-tree $T$ over an alphabet $A$ of size $k$ and a coloring $f : T \to \ell$. Let $w$ be the generating variable word of $T$ and consider the combinatorial $X$-space $S$ corresponding to the set of leaves of the tree~$T$. A $1$-dimensional subspace of $S$ corresponds to an instantiation of all the kind of variables in $w$ except one, so define a coloring $g$ that takes any $1$-dimensional subspace of $S$ to the color (by $f$) of its generating variable word cut before the first apparition of its only kind of variables. By definition of $GR(k,d+1,1,\ell)$, there is a $(d+1)$-dimensional subspace $S'$ of $S$, all whose $1$-dimensional subspaces are of the same color for $g$. Let $w'$ be the generating word of $S'$, and let $w''$ be $w'$ cut before the first occurrence of its last variable: $x_d$, let $Y$ be such that $w''$ is a $Y$-variable word ($|Y| = d$). Then $\{w''[u] : u \in A^{\leq |Y|}\}$ is an $f$-homogeneous OVW $Y$-subtree of $T$ (The $f$-color of $w''[u]$ is the $g$-color of any $1$-dimensional subspace of $S'$ whose non instantiated variable is the $|u|$-th one).
\end{proof}

Note that in the previous proof, we used the Graham-Rothschield theorem with colorings of 1-dimensional spaces to prove the Carlson-Simpson theorem which is about colorings of 0-dimensional OVW trees, that is, colorings of words. This can be generalized to prove a higher order version of the finite Carlson Simpson theorem with primitive recursive bounds (see Dodos and Kanellopoulos~\cite[Theorem 4.21]{dodos2016ramsey}). However, only the 0-dimensional version will be used in this article.

\begin{definition}
A set $X \subseteq_{fin} \NN$ is said to be $\omega^r \cdot s\mbox{-large}(\theta, \OVW)$ if for every $\ell,k < \min X$, every OVW $X$-tree $T$ over an alphabet $A$ of size $k$ and every coloring $f : T \to \ell$, there is an $\omega^r \cdot s\mbox{-large}(\theta)$ subset $Y \subseteq X$ and an $f$-homogeneous OVW $Y$-subtree of $T$.

\end{definition}

Note that by definition of an ordered $Y$-variable sequence~$w$, $|w| \geq \max Y$.


\begin{lemma}\label[lemma]{lem:cs-omega0min-large}
There exists some $n_1 \in \omega$ such that $\RCA_0$ proves that is $X$ is $\omega^{n_1}\mbox{-large}(\theta)$ then $X$ is $\omega^0\cdot(\min X - 1)\mbox{-large}(\theta,\OVW)$.  
\end{lemma}

\begin{proof}
For $X$ to be $\omega^0\cdot(\min X - 1)\mbox{-large}(\theta,\OVW)$, it is sufficient to have $|X| \geq CS(\min X - 1, \min X, \min X - 1)$.
By \Cref{thm:carlson-simpson-finitaire}, the function $x \mapsto CS(x,x+1,x)$ is primitive recursive, 
so by \Cref{lem:fast-growing-size}, there exists some~$n_1$ such that if $X$ is $\omega^{n_1}\mbox{-large}(\theta)$, then $|X| \geq CS(\min X - 1, \min X, \min X - 1)$.



\end{proof}

\begin{lemma}\label[lemma]{lem:ovw-large-preind}
$\RCA_0$ proves that for every $b,r \in \NN$, if the following statement holds:
\begin{itemize}
    \item \qt{Every $\omega^b\mbox{-large}(\theta)$ sparse set is $\omega^r\mbox{-large}(\theta, \OVW)$}
\end{itemize}
then the following statement holds: 
\begin{itemize}
    \item \qt{Every $\omega^{2b+n_0+3}$-large$(\theta)$ sparse set $X$ is $\omega^r\cdot (\min X - 1)\mbox{-large}(\theta, \OVW)$}.
\end{itemize} 
\end{lemma}

\begin{proof}
Fix $r, b \in \NN$ such that the first statement holds.
Let $X$ be $\omega^{2b+n_0+3}\mbox{-large}(\theta)$ and sparse. Let us prove that $X$ is $\omega^r\cdot (\min X - 1)\mbox{-large}(\theta, \OVW)$: \\

Consider $\ell,k < \min X$, an OVW $X$-tree $T$ over an alphabet $A$ of size $k$ and a coloring $f : T \to \ell$.
The set $X \setminus \min X$ is $\omega^{2b+n_0+2}\mbox{-large}(\theta)$, so there exists pairwise $\theta$-apart $\omega^{2b+n_0+1}\mbox{-large}(\theta)$ subsets of $X \setminus (\min X \cup \min (X \setminus \min X))$, $X_0 < \dots < X_a$ with $a = \ell \times (\min X - 2) < (\min X)^2 < \min (X \setminus \{\min X\})$ (by sparsity). For simplicity, we will assume that $T$ is an OVW $(X_0 \cup \dots \cup X_{a})$-tree (this can always be done by taking a subtree of the original tree). \\


Above every $\sigma \in A^{\min X_{a}} \cap T$ is an OVW $X_a$-subtree of $T$. The number $n_a$ of such subtrees satisfies $n_a = k^{|X_0| + \dots + |X_{a-1}|}$ (there is one subtree for each instantiation of the variables of $X_0 \cup \dots \cup X_{a-1}$), so $n_a \leq k^{\max X_{a-1}}$.

The coloring $f$ induce a coloring on each of those trees, and since all these trees are isomorphic (not only isomorphic as trees, but they share the same structure as OVW trees: their generating word only differ for indexes less than $|\sigma|$), consider the product coloring $f' : T' \to \ell^{n_a}$ of all those coloring (for $T'$ an arbitrary such $X_a$-subtree). By sparsity of $X$, $\ell^{n_a} \leq \ell^{k^{\max X_{a-1}}} < \min X_{a}$.

Since $X_a$ is $\omega^{2b+n_0+1}\mbox{-large}(\theta)$ and sparse, then it is $\omega^{r}\mbox{-large}(\OVW, \theta)$, so there exists some $\omega^r\mbox{-large}(\theta)$ subset $Y_a \subseteq X_a$ and some $f'$-homogeneous OVW $Y_a$-subtree $S' \subseteq T'$. Therefore, by definition of $f'$, there exists an instantiation of the variables in $X_a \setminus Y_a$ making $f$-homogeneous the corresponding OVW $Y_a$-subtrees of $T$ above each $\sigma \in A^{\min X_{a}} \cap T$.
Let $T'_a$ be the corresponding OVW $(X_0 \cup \dots \cup X_{a-1} \cup Y_a)$-subtree of $T$. $T'_a$ has the property that above every $\sigma \in T'_a \cap A^{\min X_{a}}$, there is a color $c^{\sigma}_a \in \ell$ such that all the $\tau$ in the OVW $Y_a$-tree above $\sigma$ have that same color $c^{\sigma}_a$. \\

Let $T_a = T \cap A^{\leq \min X_{a}}$. Note that $T_a$ is an OVW $(X_0 \cup \dots \cup X_{a-1})$-subtree of $T$. Let $S_a$ be the corresponding combinatorial $(X_0 \cup \dots \cup X_{a-1})$-space, every element of $S_a$ correspond to one instantiation of the variables of $X_0 \cup \dots \cup X_{a-1}$ and therefore to exactly one of the OVW $X_a$-subtrees considered before. So, consider the coloring $f_a : S_a \to \ell$ that send any element of $S_a$ to the color $c^{\sigma}_a$ of its corresponding OVW $X_a$-subtree. 

There are at most $k^{|X_0 \cup \dots \cup X_{a-2}|} \leq k^{\max X_{a-2}}$ combinatorial $X_{a-1}$-subspaces of $S_a$ (one for each instantiation of the variables in $X_0 \cup \dots \cup X_{a-2}$) and $f_a$ induces a coloring on each of them. So, by using the same trick of considering the product coloring, we can apply \Cref{lem:gr-largeness} to get an $\omega^{b}\mbox{-large}(\theta)$ subset $Z_{a-1} \subseteq X_{a-1}$ such that there exists an instantiation of the variables in $X_{a-1} \setminus Z_{a-1}$ making $f_a$-homogeneous the corresponding combinatorial $Z_{a-1}$-subspaces of each combinatorial $X_{a-1}$-subspaces of $S_a$.

Since $Z_{a-1}$ is $\omega^{b}\mbox{-large}(\theta)$, then it is $\omega^{r}\mbox{-large}(\OVW, \theta)$, so there exists some $\omega^{r}\mbox{-large}(\theta)$ subset $Y_{a-1} \subseteq Z_{a-1}$ and some instantiation of the variables in $Z_{a-1} \setminus Y_{a-1}$ making $f$-homogeneous the corresponding OVW $Y_{a-1}$-subtrees of $T_a$ above each $\sigma \in A^{\min X_{a-1}} \cap T_a$ (again by considering a product coloring).

Let $T'_{a-1}$ be the corresponding OVW $(X_0 \cup \dots \cup X_{a-2} \cup Y_{a-1} \cup Y_{a})$-subtree of $T$. $T'_{a-1}$ has the property that for every $\sigma \in T'_{a-1} \cap A^{\min X_{a-1}}$ there is a tuple of color $(c^{\sigma}_{a-1}, c^{\sigma}_a)$ such that all the $\tau$ in the OVW $(Y_{a-1} \cup Y_{a})$-tree above $\sigma$ have color $c^{\sigma}_{a-1}$ if $|\tau| \in Y_{a-1}$ (by construction of $Y_{a-1}$) and color $c^{\sigma}_a$ if $|\tau| \in Y_a$ (by construction of $Z_{a-1}$). \\

We can then iterate the construction and obtain a sequence of $\omega^{r}\mbox{-large}(\theta)$ subsets $Y_{i} \subseteq X_{i}$ and of OVW $(X_0 \cup \dots \cup X_{i-1} \cup Y_i \cup \dots \cup Y_{a})$-subtrees $T'_i$ of $T$, where each $T'_i$ has the property that for every $\sigma \in T'_{i} \cap A^{\min X_i}$ there is a $(a-i+1)$-uple of colors $(c^{\sigma}_{i}, \dots, c^{\sigma}_a)$ such that all the $\tau$ in the OVW $(Y_{i} \cup \dots \cup Y_a)$-tree above $\sigma$ have color $c^{\sigma}_j$ if $|\tau| \in Y_j$. At each step of the construction, we consider the tree $T_i = T \cap A^{\leq\min X_i}$, the corresponding combinatorial subspace $S_i$, a coloring $f_i : S_i \to \ell^{a+1-i}$ (The coloring $f_i$ gives the tuple of colors corresponding to the OVW $(Y_{i+1} \cup \dots \cup Y_a)$-subtree above every element of $S_i$), an $f_i$-homogeneous combinatorial $Z_i$-subspace of $S_i$ for some $Z_i \subseteq X_i$ $\omega^{b}\mbox{-large}(\theta)$, and finally an $\omega^{r}\mbox{-large}(\theta)$ subset $Y_i \subseteq Z_i$ satisfying the desired properties and $T'_i$, the corresponding OVW $(X_0 \cup \dots \cup X_{i-1} \cup Y_i \cup \dots \cup Y_{a})$-subtree of $T$. \\

Therefore, the color of every $\tau \in T'_0$ is determined by the index $j \leq a$ such that $|\tau| \in Y_j$. By the finite pigeonhole principle, there is one color that appear at least $(\min X - 1)$ times, since $a+1 = \ell \times (\min X - 2) + 1$. Let $i_0 < \dots < i_{\min X - 2} \leq a$ be indices witnessing the pigeonhole principle, and let $Y = Y_{i_1} \cup \dots \cup Y_{i_{\min X - 2}}$. The set $Y$ is $\omega^{r}\cdot(\min X - 1)\mbox{-large}(\theta)$ and any OVW $Y$-subtree of $T'_0$ is $f$-homogeneous.
\end{proof}

\begin{lemma}\label[lemma]{lem:ovw-largeness-ind}
$\RCA_0$ proves that for every $b,r \in \NN$, if the following statement holds:
\begin{itemize}
    \item \qt{Every $\omega^b\mbox{-large}(\theta)$ sparse set~$X$ is $\omega^r\cdot (\min X  - 1)\mbox{-large}(\theta, \OVW)$}
\end{itemize}
then the following statement holds: 
\begin{itemize}
    \item \qt{Every $\omega^{2b+n_0+2}$-large$(\theta)$ sparse set is $\omega^{r+1}\mbox{-large}(\theta, \OVW)$}.
\end{itemize} 
\end{lemma}

\begin{proof}
Fix $r, b \in \NN$ such that the first statement holds. Let $X$ be $\omega^{2b+n_0+2}\mbox{-large}(\theta)$ and sparse. Let us prove that $X$ is $\omega^{r+1}\mbox{-large}(\theta, \OVW)$: \\

Let $X$ be $\omega^{2b+n_0+2}\mbox{-large}(\theta)$ and sparse. 
Consider $\ell,k < \min X$, an OVW $X$-tree $T$ over an alphabet $A$ of size $k$ and a coloring $f : T \to \ell$.   

Let $X_0< \dots < X_{\ell}$ be $\omega^{2b+n_0+1}\mbox{-large}(\theta)$ subsets of $X$. 
By the same construction as in the proof of \Cref{lem:ovw-large-preind}, but replacing $\omega^r$-largeness$(\theta, \OVW)$
by $\omega^r\cdot (\min X  - 1)$-largeness$(\theta, \OVW)$, we obtain for each~$i \leq \ell$ an  $\omega^r\cdot (\min X  - 1)$-large$(\theta)$ subset
$Y_i \subseteq X_i$, and an OVW $(Y_0 \cup \dots \cup Y_{\ell})$-tree $T'_0$ such that the color of every $\rho \in T'_0$ is determined by the index $j \leq \ell$ such that $|\rho| \in Y_j$. Since $\ell + 1 > \ell$, there are indexes $i < j$ such that all the $\rho \in T'_0$ with $|\rho| \in Y_i \cup Y_j$ have the same color. Take some $y \in Y_i$, since $y \leq \min Y_j - 1$ and $Y_j$ is $\omega^r\cdot(\min X_j - 1)$-large$(\theta)$, then $\{y\} \cup Y_j$ is $\omega^{r+1}$-large$(\theta)$ and any OVW $(\{y\} \cup Y_j)$-subtree of $T'_0$ is $f$-homogeneous. 
\end{proof}

\begin{lemma}\label[lemma]{lem:ovw-largeness-goodind}
$\RCA_0$ proves that for every $b,r \in \NN$, if the following statement holds:
\begin{itemize}
    \item \qt{Every $\omega^b\mbox{-large}(\theta)$ sparse set is $\omega^r\mbox{-large}(\theta, \OVW)$}
\end{itemize}
then the following statement holds: 
\begin{itemize}
    \item \qt{Every $\omega^{4b+3n_0+8}$-large$(\theta)$ sparse set is $\omega^{r+1}\mbox{-large}(\theta, \OVW)$}.
\end{itemize} 
\end{lemma}
\begin{proof}
Immediate by \Cref{lem:ovw-large-preind} and \Cref{lem:ovw-largeness-ind}.
\end{proof}

\begin{corollary}\label[corollary]{cor:largeness-OVW}
$\RCA_0$ proves that for every $b \in \NN$,
every $\omega^{O(4^b)}$-large$(\theta)$ sparse set is $\omega^b$-large$(\theta, \OVW)$.
\end{corollary}
\begin{proof}
The $O(4^b)$ notation can actually be replaced by a complicated explicit bound $p(b)$, computed from the bounds in \Cref{lem:cs-omega0min-large} and \Cref{lem:ovw-largeness-ind}.
One can prove the explicit bound version of the statement by a simple $\Pi^0_1$-induction over~$b$, using \Cref{lem:cs-omega0min-large} and \Cref{lem:ovw-largeness-ind} for the base case,
and \Cref{lem:ovw-largeness-goodind} for the induction step.
The inductive formula is of the form \qt{for every finite set $F$ which is $\omega^{p(b)}$-large$(\theta)$ and sparse, $F$ is $\omega^b$-large$(\theta, \OVW)$}. Since being $\omega^{p(b)}$-large$(\theta)$, sparse and $\omega^b$-large$(\theta, \OVW)$ are $\Delta^0_0$, the overall formula is $\Pi^0_1$.

\end{proof}

\section{Consequences and open questions}\label[section]{sect:consequences}

We now turn to the proof of \Cref{thm:OVW-pi04-conservation} using \Cref{cor:largeness-OVW}.
The proof is an adaptation of Le Houérou, Levy Patey and Yokoyama~\cite[Proposition 2.13]{houerou2023conservation},
which is itself an elaboration of the original construction by Kirby and Paris~\cite{Kirby1977InitialSO} of a semi-regular cut.

\begin{reptheorem}{thm:OVW-pi04-conservation}[]
$\WKL_0 + \OVW$ is $\forall \Pi^0_4$-conservative over $\RCA_0 + \BSig_2$.      
\end{reptheorem}

\begin{proof}
Let $\phi(X\uh_z,t,x,y,z)$ be a $\Delta_0^0$-formula and assume that $\RCA_0 + \BSig_2 \not \vdash \forall X \forall t \exists x \forall y \exists z \phi(X\uh_z,t,x,y,z)$. In what follows, we enrich the language with two constant symbols $c$ and $C$, of first and second order, respectively.

By completeness and the Löwenheim-Skolem theorem, there exists some countable model $$\M = (M,S,c^\M, C^\M) \models \RCA_0 + \BSig_2 + \forall x \exists y \forall z \neg \phi(C\uh_z,c,x,y,z)$$
Let $\theta(x,y,z)$ be the $\Delta_0^{0,C,c}$ formula $\neg \phi(C\uh_z,c,x,y,z)$. \\

By \Cref{prop:largeness-bsig2-largeness} and \Cref{lem:large-implies-sparse}, for all standard $n$, there exists some sparse $\omega^n$-large$(\theta)$ subset of $M$. So, by compactness, we can assume that $\M$ contains a sparse $\omega^{d}$-large$(\theta)$ set $X$ for some non-standard integer $d$. \\

By a standard argument, we can consider a decreasing sequence $X = X_0 \supseteq X_1 \supseteq \dots$ such that for every $i < \omega$:

\begin{itemize}
\item[(1)] $X_i$ is $\omega^{d_i}$-large$(\theta)$ for some $d_i$ non-standard.
\item[(2)] $\min X_{i+1} > \min X_i$
\item[(3)] For every $M$-finite set $E$ such that $|E| < \min X_i$, there exists some $j > i$ such that $[\min X_j, \max X_j] \cap E = \emptyset$. 
\item[(4)] For every coloring $f : A^{<M} \to k$ for some $k < \min X_i$ and some $M$-finite alphabet $A$ with $|A| < \min X_i$ with $f \in \M$, there exists some $j > i$ and an $f$-homogeneous $\OVW$ $X_j$-subtree of $A^{<M}$.
\item[(5)] There exists some $j > i$ such that 
$$\forall x < \min X_{j} \exists y < \min X_{j+1} \forall z < \max X_{j+1} \theta(x,y,z)$$

\end{itemize}

Then, consider $I = \sup \{\min X_i | i \in \omega\}$. $I$ is a semi-regular cut of $M$ by $(3)$, therefore $(I, \mathsf{Cod}(M/I)) \models \WKL_0$. The constraint $\min X \geq c$ imposed on largeness$(\theta)$ ensures that $c$ is in $I$.

The condition $\min X_{i+1} > \min X_i$ implies that every $X_i \cap I$ is infinite in $I$.

Let $k \in I$ and $A$ a finite alphabet with $|A| \in I$, let $f : A^{<I} \to k$ be a coloring in $(I, \mathsf{Cod}(M/I))$. There exists some coloring $g : A^{<M} \to k$ in $\M$ such that $f = g \cap I$ and let $i$ such that $k, |A| < \min X_i$. By construction, let $j > i$ and $T$ a $g$-homogeneous $\OVW$ $X_j$-subtree of $A^{<M}$, then $T \cap I$ is an $f$-homogeneous $\OVW$ $X_j \cap I$-subtree of $A^{<I}$. Since $X_j \cap I$ is infinite in $I$, $(I, \mathsf{Cod}(M/I)) \models \OVW$.

Finally, $(I, \mathsf{Cod}(M/I),c^{\M},C^{\M}\cap I) \models \forall x \exists y \forall z \theta(x,y,z)$ 
as for every $k \in I$, there exists an index $i \in \omega$ such that $k < \min X_i$ and therefore by (5), there is some~$j > i$ such that $\forall x < k \exists y < \min X_{j+1} \forall z < \max X_{j+1} \ \theta(x,y,z)$, so $(I, \mathsf{Cod}(M/I), c^{\M},C^{\M} \cap I) \models \forall x < k \exists y \forall z \ \theta(x,y,z)$ (since $\max X_{j+1} > I$). 

Hence $(I, \mathsf{Cod}(M/I)) \models \WKL_0 + \OVW + \neg \forall X \forall t \exists x \forall y \exists z \phi(X \uh_z,t,x,y,z)$.

\end{proof}

\begin{corollary}\label[corollary]{cor:ovw-pi03-rca}
$\WKL_0 + \OVW$ is $\forall \Pi^0_3$-conservative over~$\RCA_0$ and $\Pi^0_2$-conservative over~$\mathsf{PRA}$.
\end{corollary}
\begin{proof}
By a parameterized version of the Parsons, Paris and Friedman conservation theorem (see~\cite{hajek1998metamathematics,kaye1991models}), $\RCA_0 + \BSig_2$ is $\forall \Pi^0_3$-conservative over~$\RCA_0$. By Friedman~\cite{friedmancom} (see \cite{simpson2009subsystems}), $\RCA_0$ is $\Pi^0_2$-conservative over~$\mathsf{PRA}$.
\end{proof}

\begin{corollary}\label[corollary]{cor:ovw-separation-aca}
$\WKL_0 + \OVW$ does not imply~$\ACA_0$.
\end{corollary}
\begin{proof}
By Simpson~\cite[Corollary VIII.1.7]{simpson2009subsystems}, $\ACA_0$ proves the consistency of~$\RCA_0$, which is a $\Pi^0_1$ statement, while $\RCA_0$ does not by the second G\"odel incompleteness theorem. Thus, by \Cref{cor:ovw-pi03-rca}, $\WKL_0 + \OVW$ does not imply the consistency of~$\RCA_0$.
\end{proof}

\subsection{Open questions}

As mentioned in the introduction, the tree theorem for singletons ($\TT^1$) is strictly stronger than~$\BSig_2$. Since~$\TT^1$ is a $\Pi^1_2$ statement, this does not rule out the possibility that $\RCA_0 + \TT^1$ is $\Pi^1_1$-conservative over~$\RCA_0 + \BSig_2$. The same question can be asked about the Ordered Variable Word theorem.

\begin{question}\label[question]{ques:ovw-bsig2}
Is $\RCA_0 + \OVW$ $\Pi^1_1$-conservative over~$\RCA_0 + \BSig_2$?
\end{question}

The Tree Theorem for singletons and Milliken's tree theorem for singletons are both provable over~$\RCA_0 + \ISig_2$.
On the other hand, $\OVW$ is not computably true, as it admits a computable instance with no $\Delta^0_2$ solutions.

\begin{question}\label[question]{ques:ovw-isig2}
Is $\RCA_0 + \ISig_2 + \OVW$ $\Pi^1_1$-conservative over~$\RCA_0 + \ISig_2$?
\end{question}

If the answer to \Cref{ques:ovw-isig2} is positive, then by Fiori-Carones et al.~\cite{fiori2021isomorphism}, 
\Cref{ques:ovw-bsig2} can be reduced to whether $\RCA_0 + \OVW$ is $\forall \Pi^0_5$-conservative over~$\RCA_0 + \BSig_2$.

The proof of \Cref{cor:ovw-separation-aca} involves the construction of a non-standard model of~$\WKL_0 + \OVW$. It is natural to wonder whether such separation also holds over $\omega$-models, that is, models whose first-order part consists of the standard integers.
A problem $\mathsf{P}$ admits \emph{cone avoidance} if for every non-computable set~$C$ and every computable $\mathsf{P}$-instance~$X$, there exists a $\mathsf{P}$-solution $Y$ to~$X$ such that $C \not \leq_T Y$. If~$\mathsf{P}$ admits cone avoidance, then there exists an $\omega$-model of $\WKL_0 + \mathsf{P}$ which does not contain~$\emptyset'$, hence is not a model of~$\ACA_0$. All the known separations of problems from $\ACA_0$ are done by proving that $\mathsf{P}$ admits cone avoidance.

\begin{question}
Does $\OVW$ admits cone avoidance?
\end{question}


\bigskip
\begin{center}
\textbf{Acknowledgements}
\end{center}
The authors are thankful to Keita Yokoyama for insightful comments and discussions
and to the anonymous reviewer for improvement suggestions.

\bibliographystyle{plain}
\bibliography{biblio}

\end{document}